\documentclass[a4paper, reqno]{amsart}
\newif\iffinal

\long\def\comment#1{\relax}

\usepackage[utf8x]{inputenc}
\usepackage[T1]{fontenc}
\usepackage{lmodern}
\usepackage[english]{babel}

\usepackage{tikz}
\usetikzlibrary{positioning}
\usepackage{float}
\usepackage{amssymb}
\usepackage{colonequals}
\usepackage{mathtools}
\usepackage{hyperref}
\usepackage{colortbl}
\usepackage{color}

\DeclareMathOperator{\Int}{Int}
\DeclareMathOperator{\divides}{\big\vert}

\let\natn=\N

\newcommand{\primes}{\mathbb{P}}

\newcommand{\Z}{\mathbb{Z}}

\newcommand{\IntZ}{\Int(\Z)}
\newcommand{\IntD}{\Int(D)}

\newcommand{\fixdiv}{\mathsf{d}}
\let\fd=\fixdiv
\newcommand{\vp}{v_{\scriptscriptstyle p}}
\newcommand{\fdp}{\mathsf{d}_{\scriptscriptstyle p}}

\def\card#1{\left|#1\right|}

\newtheorem{theorem}{Theorem}
\newtheorem{proposition}{Proposition}[section]
\newtheorem{corollary}[proposition]{Corollary}

\newtheorem{lemma}[proposition]{Lemma}
\theoremstyle{definition}
\newtheorem{definition}[proposition]{Definition}

\newtheorem{example}[proposition]{Example}
\newtheorem{remark}[proposition]{Remark}

\tikzset{main node/.style={circle,draw,minimum size=0.4cm,inner sep=0pt},}
\author{Sophie Frisch}
\address{\parbox{\linewidth}{Institut für Analysis und Zahlentheorie,
  Graz University of Technology\\
    Kopernikusgasse 24, 8010 Graz, Austria}}
\email{\href{mailto:frisch@math.tugraz.at}{frisch@math.tugraz.at}}
\thanks{S.~Frisch is supported by the Austrian Science Fund (FWF):
P~30934}

\author{Sarah Nakato}
\address{\parbox{\linewidth}{Institut für Analysis und Zahlentheorie,
  Graz University of Technology\\
    Kopernikusgasse 24, 8010 Graz, Austria}}
\email{\href{mailto:snakato@tugraz.at}{snakato@tugraz.at}}
\thanks{S.~Nakato is supported by the Austrian Science Fund (FWF):
P~30934}

\title[Graph-theoretic criterion for absolute irreducibility]%
{A graph-theoretic criterion for absolute irreducibility of
integer-valued polynomials with square-free denominator}

\keywords{factorization, non-unique factorization,
irreducible elements, absolutely irreducible elements, atoms,
strong atoms, atomic domains, integer-valued polynomials,
simple graphs, connected graphs}

\subjclass[2010]{13A05, 13B25, 13F20, 11R09, 11C08, 13P05}

\begin{document}

\begin{abstract}
An irreducible element of a commutative ring is absolutely irreducible
if no power of it has more than one (essentially different) factorization
into irreducibles. In the case of the ring
$\IntD=\{f\in K[x]\mid f(D)\subseteq D\}$, of integer-valued polynomials
on a principal ideal domain $D$ with quotient field $K$,
we give an easy to verify graph-theoretic sufficient condition for an
element to be absolutely irreducible and show a partial converse:
the condition is necessary and sufficient
for polynomials with square-free denominator.
\end{abstract}

\maketitle

\section{Introduction}
An intriguing feature of non-unique factorization (of elements of an
integral domain into irreducibles) is the existence of non-absolutely
irreducible elements, that is, irreducible elements some of whose powers
allow several essentially different factorizations into irreducibles
\cite{BaKr10HFKR, GeHa06NUF, Ka81CI, Na19NAB, Rd83ANF}.

For rings of integers in number fields, their existence actually
characterizes non-unique factorization, as
Chapman and Krause~\cite{ChKr12AD} have shown.

Here, we investigate absolutely and non-absolutely irreducible elements
in the context of non-unique factorization into irreducibles in the ring
of integer-valued polynomials on $D$
\[
\IntD=\{f\in K[x]\mid f(D)\subseteq D\},
\]
where $D$ is a principal ideal domain and $K$ its quotient field.

In an earlier paper \cite[Remark 3.9]{FrNaRi19SLF}
we already hinted at a graph-theoretic sufficient condition for
$f\in\IntD$ to be irreducible. We spell this out more fully in
Theorem~\ref{sufficient-for-irreducibility}. This condition
is not, however, necessary.

We formulate a similar graph-theoretic sufficient condition for
$f\in\IntD$ to be absolutely irreducible in
Theorem~\ref{connectedness-sufficient}, and show a partial converse.
Namely, our criterion for absolute irreducibility is necessary and
sufficient in the special case of polynomials with square-free
denominator, cf.~Theorem~\ref{criterion-for-squarefree-abs-irred}.

First, we recall some terminology.
Let $R$ be a commutative ring with identity.
\begin{enumerate}
\item $r\in R$ is called \emph{irreducible} in $R$
(or, an \emph{atom} of $R$) if it is a non-zero non-unit that is
not a product of two non-units of $R$.
\item A \emph{factorization} (into irreducibles) of $r$ in $R$ is an
expression
       \begin{equation}
       \label{eq:fac} r = a_{1}\cdots a_{n}
       \end{equation}
       where $n\ge 1$ and $a_i$ is irreducible in $R$ for $1\le i \le n$.
\item $r,s\in R$ are \emph{associated} in $R$ if there exists
a unit $u \in R$ such that $r = us$. We denote this by $r \sim s$.
\item Two factorizations into irreducibles of the same element,
       \begin{equation}\label{eq:2-fac-same-diff}
       r = a_{1}\cdots a_{n} = b_{1} \cdots b_{m},
       \end{equation}
are called \emph{essentially the same} if $n = m$ and, after a suitable
re-indexing,
$a_{j}\sim b_{j}$ for $1 \leq j \leq m$.
Otherwise, the factorizations in
\eqref{eq:2-fac-same-diff} are called \emph{essentially different}.
\end{enumerate}

\begin{definition}\label{defabsirred}
Let $R$ be a commutative ring with identity.
An irreducible element $c\in R$ is called \emph{absolutely irreducible}
(or, a \emph{strong atom}), if for all natural numbers $n$, every factorization of $c^n$
 is essentially the same as $c^n = c \cdots c$.
\end{definition}

Note the following fine distinction: an element of $R$ that is called
\emph{``not absolutely irreducible''} might not be irreducible at all,
whereas a \emph{``non-absolutely irreducible''} element is assumed to
be irreducible, but not absolutely irreducible.

We now concentrate on integer-valued polynomials over a principal
ideal domain.

Recall that a polynomial in $D[x]$, where $D$ is a principal ideal
domain, is called \emph{primitive} if the greatest common divisor
of its coefficients is $1$.

\begin{definition}\label{defstandardform}
Let $D$ be a principal ideal domain with quotient field $K$, and
$f\in K[x]$ a non-zero polynomial. We write $f$ as
\[
f = \frac{a \prod_{i \in I}g_i}{b},
\]
where $a,b\in D\setminus \{0\}$ with $\gcd(a,b)=1$, $I$ a finite
(possibly empty) set, and each $g_i$ primitive and irreducible in
$D[x]$ and call this the \emph{standard form} of $f$.

We refer to $b$ as the \emph{denominator}, to $a$ as the
\emph{constant factor}, and to $a \prod_{i \in I}g_i$ as
the \emph{numerator} of $f$, keeping in mind that each of
them is well-defined and unique only up to multiplication by
units of $D$.
\end{definition}

\begin{definition}\label{deffd}
For $f\in\IntD$, the \emph{fixed divisor} of
$f$, denoted $\fd(f)$, is the ideal of $D$ generated by $f(D)$.

An integer-valued polynomial $f\in\IntD$
with $\fd(f)=D$ is called \emph{image-primitive}.

When $D$ is a principal ideal domain, we may, by abuse of
notation, write the generator for the ideal, as in $\fd(f)=c$
meaning $\fd(f)=cD$.
\end{definition}

\begin{remark}\label{remstandardint}
Let $D$ be a principal ideal domain with quotient field $K$,
and $f\in K[x]$ written in standard form as in
Definition~\ref{defstandardform}. Then $f$ is in $\IntD$
if and only if $b$ divides $\fd({ \prod_{i \in I}g_i})$.
\end{remark}

\begin{remark}\label{remimageprimitive}
Let $D$ be a principal ideal domain with quotient field $K$.
Then any non-constant
irreducible element of $\IntD$ is necessarily image-primitive.
Otherwise, if a prime element $p \in D$ divides $\fd(f)$, then
\[f = p\cdot \frac{f}{p}\] is a non-trivial factorization of $f$.

Furthermore, $f\in K[x] \setminus\{0\}$ (written in standard form
as in Definition~\ref{defstandardform}) is an image-primitive element
of $\IntD$ if and only if (up to multiplication by units) $a=1$ and
$b=\fd({ \prod_{i \in I}g_i})$.
\end{remark}

\begin{definition}\label{deffdp}
Let $D$ be a principal ideal domain. For $f\in\IntD$, and $p$
a prime element in $D$, we let \[\fdp(f)=\vp(\fd(f))\]
\end{definition}

\begin{remark}\label{remfdnotmult}
By the above definition,
\[
\fd(f) = \prod_{p\in \primes} p^{\fdp(f)}
\quad\textrm{and}\quad
\fdp(f)=\min_{c\in D} \vp(f(c))
\]
where $\primes$ is a set of representatives of the prime elements of
$D$ up to multiplication by units.

By the nature of the minimum function, the fixed divisor is not
multiplicative:
\[\fdp(f) + \fdp(g) \le \fdp(fg),\]
but the inequality may be strict. Accordingly,
\[\fd(f)\fd(g) \divides \fd(fg),\]
 but the division may be strict. Note, however, that
\[\fd(f^n)=\fd(f)^n\]
for all $f\in\IntD$ and $n\in\natn$.
\end{remark}

\section{Graph-theoretic irreducibility criteria}
We refer to, for instance, \cite{BoMu2008GT} for the graph theory
terms we use in this section.
\begin{definition}\label{defessential}
Let $D$ be a principal ideal domain, $I \neq \emptyset$ a finite set
and for $i \in I$, let $g_i \in D[x]$ be non-constant and primitive.
Let $g(x)=\prod_{i\in I}g_i$, and $p\in D$ a prime.
\begin{enumerate}
\item\label{edef}
We say that $g_i$ is \emph{essential} for $p$ among the $g_j$
with $j\in I$ if $p\divides \fixdiv(g)$ and
there exists a $w\in D$ such that $v_p(g_i(w))>0$
and $v_p(g_j(w))=0$ for all $j\in I\setminus\{i\}$.
Such a $w$ is then called a witness for
$g_i$ being essential for $p$.
\item\label{qdef}
We say that $g_i$ is \emph{quintessential} for $p$ among the
$g_j$ with $j\in I$ if $p \divides \fixdiv(g)$ and
there exists  $w\in D$ such that
$v_p(g_i(w))=v_p(\fixdiv(g))$ and
$v_p(g_j(w))=0$ for all $j\in I\setminus\{i\}$.
Such a $w$ is called a witness for $g_i$ being quintessential for $p$.
\end{enumerate}
We will omit saying ``among the $g_j$ with $j\in I$''
if the indexed set of polynomials is clear from the context.
\end{definition}

\begin{remark}\label{essentialunique}
When we consider an indexed set of polynomials $g_i$ with $i\in I$,
we are not, in general, requiring $g_i\ne g_j$ for $i\ne j$. Note,
however, that $g_i$ being essential (among the $g_j$ with $j\in I$)
for some prime element $p \in D$ implies $g_i\not\sim g_j$ in $D[x]$
for all $j\in I\setminus\{i\}$.
\end{remark}

\begin{definition}\label{defgraph}
Let $D$ be a principal ideal domain,
$I \neq \emptyset$ a finite set and for each $i \in I$,
$g_i \in D[x]$ primitive and irreducible.
\begin{enumerate}
\item\label{egraphdef}
The \emph{essential graph} of the indexed set of polynomials
$(g_i\mid i\in I)$ is the simple undirected graph whose set of
vertices is $I$, and in which $(i,j)$ is an edge if and only if
there exists a prime element $p$ in $D$ such that both $g_i$ and $g_j$
are essential for $p$ among the $g_k$ with $k\in I$.
\item\label{qgraphdef}
The \emph{quintessential graph} of the indexed set of
polynomials $(g_i\mid i\in I)$ is the simple undirected graph
whose set of vertices is $I$, and in which $(i,j)$ is an edge if
and only if there exists a prime element $p$ in $D$ such that
both $g_i$ and $g_j$ are quintessential for $p$ among the $g_k$ with $k\in I$.
\end{enumerate}
\end{definition}
\begin{example}\label{example:graphs}
Let $I = \{1, 2, 3, 4\}$ and for $i \in I,$ $g_i \in \Z[x]$ as follows:
\[
g_1 = x^3 - 19,\quad g_2 = x^2+9,\quad g_3 = x^2 + 1,\quad g_4 = x-5,
\]
 and set
\[g = (x^3 -19)(x^2+9)(x^2 + 1)(x-5).\]
A quick check shows that the fixed divisor of $g$ is $15$.
\begin{enumerate}
\item Taking $w=1, 2, 0$ respectively, as witnesses, we see that
$g_2, g_3, g_4$ are quintessential for $5$.
The polynomial $g_1$ is not essential for $5$ because
$v_5(g_1(a)) >0$ only if $a \in 4 + 5\Z$ and for such $a$,
also $v_5(g_2(a)) >0$.
\item Taking $w=1, 0, 2$ respectively, as witnesses, we see that
$g_1, g_2, g_4$ are essential for $3$.
Only $g_4$ is quintessential for $3$.
The polynomial $g_3$ is not essential for $3$.
\end{enumerate}
Figure \ref{fig:Quint} shows the essential and quintessential
graphs of $(g_1, g_2, g_3, g_4)$.

\begin{figure}[H]
\centering
\begin{tikzpicture}[scale=0.33]
    \node[main node] (3) {$3$};
    \node[main node] (4) [right = 2cm of 3]  {$4$};
    \node[main node] (1) [below = 2cm of  3] {$1$};
    \node[main node] (2) [below  = 2cm of 4]  {$2$};
 \node at (3.5,-9.8,1) {\textit{Essential graph}};

      \path[draw,thick]
    (3) edge node {} (4)
    (1) edge node {} (2)
    (3) edge node {} (2)
    (4) edge node {} (1)
    (4) edge node {} (2);
\end{tikzpicture}
\hspace{3cm}
\begin{tikzpicture}[scale=0.33]
    \node[main node] (3) {$3$};
    \node[main node] (4) [right = 2cm of 3]  {$4$};
    \node[main node] (1) [below = 2cm of  3] {$1$};
    \node[main node] (2) [below  = 2cm of 4]  {$2$};
    \node at (3.5,-9.8,1) {\textit{Quintessential graph}};

      \path[draw,thick]
    (3) edge node {} (4)
    (1) edge node {} (1)
    (3) edge node {} (2)
    (4) edge node {} (2);
\end{tikzpicture}
\caption{Graphs for Example \ref{example:graphs}} \label{fig:Quint}
\end{figure}
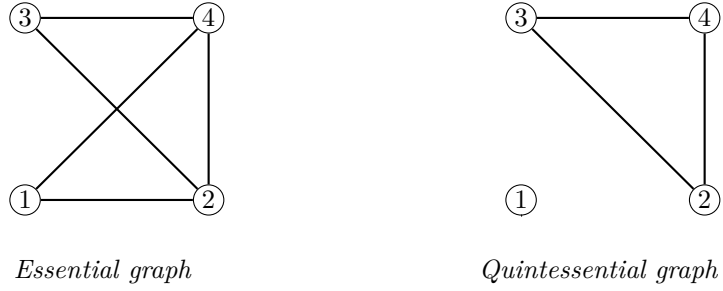
\end{example}

\begin{lemma}\label{same-exponent}
Let $D$ be a principal ideal domain and $f\in\IntD$ a non-constant
image-primitive integer-valued polynomial, written in standard form
according to Definition~\ref{defstandardform} as
\[
f = \frac{\prod_{i\in I} g_i}{\prod_{p\in T} p^{e_p}},
\]
where $T$ is a finite set of pairwise non-associated primes of $D$,
and let $n\in\natn$.

Every $h\in\IntD$ dividing $f^n$ can be written as
\[
h(x) = \frac{\prod_{i\in I} g_i^{\gamma_i(h)}}{\prod_{p\in T}p^{\beta_p(h)}},
\]
with $\gamma_i(h)\in\natn_0$ for $i\in I$ and
unique $\beta_p(h)\in\natn_0$ for $p\in T$. Moreover, every such
representation of $h$ satisfies:

\begin{enumerate}
\item\label{uniquegammaj}
If $q\in T$ and $j\in I$ such that $g_j$ is quintessential
for $q$ among the $i\in I$, then
\[
\beta_q(h) = e_q \gamma_j(h).
\]
\item\label{samegamma}
In particular, whenever $g_j$ and $g_k$ are both quintessential for
the same prime $q\in T$, then $\gamma_j(h)=\gamma_k(h)$.
\end{enumerate}
\end{lemma}

\begin{proof}
We know $\fd(f^n)=\fd(f)^n$ (cf.~Remark~\ref{remfdnotmult}). So, $f^n$
is image-primitive, and, therefore, all polynomials in $\IntD$ dividing
$f^n$ are image-primitive.  Let $f^n = hk$ with $h, k\in\IntD$.
When $h$ is written in standard form as in
Definition~\ref{defstandardform}, the fixed divisor of the
numerator equals the denominator, and the constant factor is a unit.
The same holds for $k$. This is so because $h$ and $k$ are
image-primitive; see Remark~\ref{remimageprimitive}.

Now let $q \in D$ be prime and $j\in I$ such that $g_j$ is quintessential
for $q$. Note that, by Remark~\ref{essentialunique} and unique
factorization in $K[x]$, the exponent of $g_j$ in the numerator
of any factor of $f^n$ is unique.

Writing $f^n = hk$ as
\[
\frac{\prod_{i\in I} g_i^n}{\prod_{p\in T} p^{ne_p}} =
\frac{\prod_{i\in I} g_i^{\gamma_i(h)}}{\prod_{p\in T}p^{\beta_p(h)}}\cdot
\frac{\prod_{i\in I} g_i^{\gamma_i(k)}}{\prod_{p\in T}p^{\beta_p(k)}},
\]
we observe the following equalities and inequalities of the exponents:
\begin{enumerate}
\item\label{eqone}
$ne_q = \beta_q(h) + \beta_q(k)$
\item\label{eqtwo}
$n = \gamma_j(h) + \gamma_j(k)$ and hence
$ne_q = e_q\gamma_j(h) + e_q\gamma_j(k)$
\item\label{ineqthree}
$e_q\gamma_j(h)\ge \beta_q(h)$ and $e_q\gamma_j(k) \ge \beta_q(k)$.
\end{enumerate}

\ref{eqone} follows from unique factorization in $D$.

\ref{eqtwo} follows from unique factorization in $K[x]$
and Remark~\ref{essentialunique}.

To see \ref{ineqthree}, consider a witness $w$
for $g_j$ being quintessential for $q$. Since $f$ is image-primitive,
$e_q= v_q(\fd(\prod_{i\in I}g_i))$, by Remark~\ref{remimageprimitive}. From
Definition~\ref{defessential} and Remark~\ref{remstandardint}
we deduce
\[
e_q\gamma_j(h) = v_q(g_j(w))\gamma_j(h) =
v_q(g_j^{\gamma_j(h)}(w)) =
v_q\left(\prod_{i\in I} g_i(w)^{\gamma_i(h)}\right)\ge \beta_q(h)
\]
(and similarly for $k$ instead of $h$).

Finally, \ref{eqone} - \ref{ineqthree} together
imply $e_q\gamma_j(h) = \beta_q(h)$ and $e_q\gamma_j(k)  = \beta_q(k)$.
\end{proof}

\begin{theorem}\label{sufficient-for-irreducibility}
Let $D$ be a principal ideal domain with quotient field $K$.
Let $f\in\IntD$ be a non-constant image-primitive integer-valued polynomial,
written in standard form as $f=g/b$ with $b\in D \setminus \{0\}$,
and $g=\prod_{i\in I} g_i$, where each $g_i$ is primitive
and irreducible in $D[x]$.

If the essential graph of $(g_i\mid i\in I)$ is connected, then
$f$ is irreducible in $\IntD$.
\end{theorem}

\begin{proof}
If $\card{I}=1$, then $f$ is irreducible in $K[x]$, and, by being
image-primitive, also irreducible in $\IntD$.

Now assume $\card{I}>1$, and suppose $f$ can be expressed as a
product of $m$ non-units $f=f_1\cdots f_m$ in $\IntD$. Since $\fixdiv(f)=1$,
we see immediately that no $f_i$ is a constant, and that $\fixdiv(f_k)=1$
for every $1\le k\le m$.

Write $f_k=h_k/b_k$ with $b_k\in D$ and $h_k$ primitive in $D[x]$.
Then $b=b_1\cdots b_m$ and there exists a partition of $I$ into non-empty
pairwise disjoint subsets $I=\bigcup_{i=1}^m I_k$, such that
$h_k=\prod_{i\in I_k} g_i$.

Select $i\in I_1$ and $j\in I$ with $j\ne i$. We show that also
$j\in I_1$. Let $i={i_0}, {i_1},\ldots, {i_s}=j$ be a path
from $i$ to $j$ in the essential graph of $(g_i\mid i\in I)$. For
some prime element $p_0$ in $D$ dividing $b$, $g_{i_0}$ and $g_{i_1}$ are both
essential for $p_0$.
As $g_i$ is essential for $p_0$, $p_0$ cannot divide any $b_k$ with
$k\ne 1$ and, hence, $p_0$ divides $b_1$. For any $g_k$ essential for $p_0$
it follows that $k\in I_1$, and, in particular, $i_1\in I_1$.
The same argument with reference to a prime $p_k$ for which both $g_{i_k}$ and
$g_{i_{k+1}}$ are essential, shows for any two adjacent vertices ${i_k}$ and ${i_{k+1}}$ in the path
that they pertain to the same $I_k$, and, finally, that $j\in I_1$.

As $j\in I$ was arbitrary, $I_1=I$ and $m=1$.
\end{proof}

\begin{theorem}\label{connectedness-sufficient}
Let $D$ be a principal ideal domain and $f\in\IntD$ be non-constant
and image-primitive, written in standard form as
\begin{equation*}
f = \frac{\prod_{i\in I} g_i}{\prod_{p\in T} p^{e_p} },
\end{equation*}
where $I \neq \emptyset$ is a finite set and for $i \in I$,
$g_i\in D[x]$ is primitive and irreducible in $D[x]$.

If the quintessential graph $G$ of $(g_i\mid i\in I)$ is connected,
then $f$ is absolutely irreducible.
\end{theorem}

\begin{proof}
Suppose
\begin{equation*}
f^n = \prod_{l=1}^sf_l,\quad\textrm{where}\quad
f_l = \frac{\prod_{i \in I}g_i^{m_l(i)}}{\prod_{p\in T}p^{k_l(p)}}
\end{equation*}
and $0 \leq m_l(i) \leq n$, $0 \leq k_l(p) \leq ne_p$ and for all $i$,
$\sum_{l=1}^sm_l(i) = n$
and for all $p$, $\sum_{l=1}^sk_l(p) = n e_p$.

Fix $t$ with $1 \le t\le s$. We show that $f_t$ is a power of $f$
by showing that each $g_i$ with $i\in I$ occurs in the numerator of
$f_t$ with the same exponent.

Let $i,j\in I$. By the connectedness of the quintessential graph,
there exists a sequence of indices in $I$,
$i= {i_0}, i_1, i_2,\ldots, i_k=j$ and for each $h$,
a prime element $p_h$ in $T$ such that $g_{i_h}$ and $g_{i_{h+1}}$ are both
quintessential for $p_h$. By Lemma \ref{same-exponent}, $g_{i_h}$ and
$g_{i_{h+1}}$ occur in the numerator of $f_t$ with the same exponent.
Eventually, $g_i$ and $g_j$ occur in the numerator of $f_t$ with the same
exponent, for arbitrary $i,j\in I$.
In an image-primitive polynomial, the numerator determines its
denominator (as in Remark~\ref{remimageprimitive})
and, hence, $f_t$ is a power of $f$. Since $f_t$ is irreducible, $f_t=f$.
\end{proof}

\begin{example}
The binomial polynomial
\[{x \choose p} = \frac{x(x-1) \cdots (x-p+1) }{p!}\]
where $p \in \Z$ is a prime, is absolutely irreducible in $\IntZ$,
by Theorem~\ref{connectedness-sufficient}.
\end{example}

The converse of Theorem~\ref{connectedness-sufficient} does not hold
in general.  For instance, the polynomial
\[f = \frac{x^2(x^2+3)}{4} \in \IntZ \]
is absolutely irreducible in $\IntZ$ but the quintessential graph
of $(x, x, x^2+3)$ is not connected.

There is, however, a converse to
Theorem~\ref{connectedness-sufficient} in the special case where
the denominator of $f$ is square-free, as we now proceed to show,
cf.~Theorem~\ref{criterion-for-squarefree-abs-irred}.

\section{Absolutely irreducible polynomials with square-free denominator}

Let $D$ be a principal ideal domain with quotient field $K$. When
we talk of the denominator of a polynomial in $K[x]$, this refers
to the standard form of a polynomial introduced in
Definition~\ref{defstandardform}.

\begin{remark}\label{Remark:gsdistinct}
Let $D$ be a principal ideal domain.
Suppose the denominator of $f\in\Int(D)$, written in standard form
as in Definition~\ref{defstandardform}, is square-free:
\[
f = \frac{\prod_{i \in I}g_i}{\prod_{p \in T}p} .
\]
Then, if $f$ is irreducible in $\IntD$, it follows that each $g_i$
is essential for some $p \in T$.  Otherwise, we can split off $g_i$.
This further implies $g_i \not\sim g_j$ in $D[x]$ for $i \neq j$,
whenever $f\in\Int(D)$ with square-free denominator is irreducible. A criterion for irreducibility
of an integer-valued polynomial with square-free denominator
has been given by Peruginelli \cite{Peru2015SFF}.
\end{remark}

\begin{theorem}\label{criterion-for-squarefree-abs-irred}
Let $D$ be a principal ideal domain and $f\in\IntD$ be
non-constant and image-primitive, with
square-free denominator, written in standard form as
\[
f = \frac{\prod_{i\in I} g_i}{\prod_{p\in T} p },
\]
where $I \neq \emptyset$ is a finite set and for $i \in I$,
$g_i\in D[x]$ is primitive and irreducible in $D[x]$.

Let $G$ be the quintessential graph of $(g_i\mid i\in I)$
as in Definition~\ref{defgraph}.

Then $f$ is absolutely irreducible if and only if $G$
is connected.
%
\end{theorem}

\begin{proof}
In view of Theorem~\ref{connectedness-sufficient}, we only need to
show necessity. If $\card{I}=1$, then $G$ is connected.
Now assume $\card{I}>1$, and suppose $G$ is not connected.
We show that $f$ is not absolutely irreducible. If $f$ is not even
irreducible, we are done. So suppose $f$ is irreducible. This implies
$g_i \not\sim g_j$ in $D[x]$ for $i \neq j$,
by Remark~\ref{Remark:gsdistinct}.
Since $G$ is not connected,
$I$ is a disjoint union of $J_1$ and $J_2$, both non-empty,
such that there is no edge $(i,j)$ with $i\in J_1$ and $j\in J_2$.

We express $T$ as a disjoint union of $T_1$ and $T_2$ by assigning
every $p\in T$ for which some $g_i\in J_1$ is quintessential to $T_1$,
every $p\in T$ for which some $g_i\in J_2$ is quintessential to $T_2$,
and assigning each $p\in T$ for which no $g_i$ is quintessential
to $T_1$ or $T_2$ arbitrarily.
(It may happen that $T_1 = \emptyset$ and $T_2 = T$ or vice versa).

Then $f^3$ factors in $\IntD$ as follows:
\begin{equation*}
f^3 = \frac{\left(\prod_{i\in J_1} g_i\right)^2 \prod_{j\in J_2} g_j}
{\left(\prod_{p\in T_1} p\right)^2 \prod_{q\in T_2} q} \cdot
\frac{\left(\prod_{j\in J_2} g_j\right)^2 \prod_{i\in J_1} g_i}
{\left(\prod_{q\in T_2} q\right)^2 \prod_{p\in T_1} p}.
\end{equation*}

As $\IntD$ is atomic (cf.~\cite{CaCha1995ElaIVP}), each of the two
factors above can further be factored into irreducibles. Since
$J_1$ and $J_2$ are both non-empty and $g_i \not\sim g_j$ in
$D[x]$ (and hence, $g_i \not\sim g_j$ in $K[x]$) for $i \neq j$,
it is clear that the resulting factorization of $f^3$ into irreducibles
is essentially different from $f \cdot f \cdot f$.
\end{proof}

\begin{remark}
Let $D$ be a principal ideal domain.
The proof of Theorem~\ref{criterion-for-squarefree-abs-irred} shows
that any non-absolutely irreducible element $f\in\Int(D)$ with square-free
denominator exhibits non-unique factorization of $f^n$ already for $n=3$.
\end{remark}

If $f(x) = {\prod_{i \in I}g_i(x)}/{p}$, where $D$ is a principal
ideal domain, $p$ a prime of $D$ and each $g_i\in D[x]$ primitive and
irreducible in $D[x]$, then it is easy to see that $f$ is an
irreducible element of $\IntD$
if and only if
\begin{enumerate}
\item
$\fd(\prod_{i \in I}g_i(x))=p$ and
\item
each $g_i$ is essential for $p$, that is,
for each $i \in I$ there exists
$w_i \in D$ such that $v_p(g_i(w_i))>0$ and $v_p(g_j(w_i))=0$ for
all $j\in I\setminus\{i\}$.
\end{enumerate}

An analogous statement relates absolutely irreducible integer-valued
polynomials with prime denominator
to quintessential irreducible factors of the numerator:

\begin{corollary}
Let $D$ be a principal ideal domain, $p \in D$ a prime, and
$I \neq \emptyset$ a finite set.
For $i\in I$, let $g_i\in D[x]$ be primitive and irreducible in $D[x]$. Let
\begin{equation*}
f(x) = \frac{\prod_{i \in I}g_i(x)}{p}.
\end{equation*}

Then $f$ is an absolutely irreducible element of $\IntD$ if and only if
\begin{enumerate}
\item
$\fd(\prod_{i \in I}g_i(x))=p$ and
\item
each $g_i$ is quintessential for $p$ among the $g_i$ with $i\in I$,
that is, for each $i \in I$ there exists $w_i \in D$ such that
$v_p(g_i(w_i))=1$ and $v_p(g_j(w_i))=0$ for all $j\in I\setminus\{i\}$.
\end{enumerate}
\end{corollary}

\begin{proof}
If $\fd(\prod_{i \in I}g_i(x))=p$, then $f\in\IntD$ with $\fd(f)=1$,
and Theorem~\ref{criterion-for-squarefree-abs-irred} applies.
If, on the other hand, $f$ is in $\IntD$
and is absolutely irreducible, then $f$ is, in particular, irreducible
and therefore $\fd(f)=1$, and, again,
Theorem~\ref{criterion-for-squarefree-abs-irred} applies.
Now the statement follows from the fact that, whenever
$\fd(\prod_{i \in I}g_i(x))=p$ is prime, the quintessential graph
of $(g_i\mid i\in I)$ is connected if and only if every $g_i$
is quintessential for $p$.
\end{proof}

We conclude by an example of how to apply
Theorem~\ref{criterion-for-squarefree-abs-irred}:

\begin{example}\label{example:nonabsirred}
The following polynomial $f\in\IntZ$ is irreducible, by
Theorem~\ref{sufficient-for-irreducibility}; but not
absolutely irreducible, by Theorem~\ref{criterion-for-squarefree-abs-irred}:
\[f = \frac{(x^3 -19)(x^2+9)(x^2 + 1)(x-5)}{15} \]
This is so because the essential graph of
$(x^3 -19, x^2+9, x^2 + 1, x-5)$ is
connected, but the quintessential graph is not connected,
see Example~\ref{example:graphs} and Figure~\ref{fig:Quint}.
\comment{
\[
f^3 = \frac{[(x^2+9)(x^2 + 1)(x-5)]^2(x^3 -19)}{15^2} \cdot
\frac{(x^3 -19)^2[(x^2+9)(x^2 + 1)(x-5)] }{15}
\]
\[
=\frac{[(x^2+9)(x^2 + 1)(x-5)]^2(x^3 -19)}{15^2} \cdot
\frac{(x^3 -19)(x^2+9)(x^2 + 1)(x-5)}{15} \cdot (x^3 -19)
\]
is a factorization of $f^3$ essentially different from $f\cdot f \cdot f$.
}
\end{example}

\nocite{Peru2015SFF}
\bibliographystyle{plain}
\bibliography{bib_Absolutely_irreducibles}

\end{document}